\documentclass[12pt]{article}
\usepackage{latexsym, amsmath, amsfonts, amsthm, amssymb}
\usepackage{times}
\usepackage{a4wide}
\usepackage{times}
\usepackage{graphicx, tocloft}
\usepackage{mathrsfs}

\def \RR {\mathbb R}

\def \EE {\mathbb E}

\def \eps {\varepsilon}
\def \vphi {\varphi}

\def \cV {\mathcal V}

\def \cF {\mathcal F}

\def \cP {\mathcal P}
\def \cN {\mathcal N}

\newtheorem{theorem}{Theorem}[section]

\newtheorem{lemma}[theorem]{Lemma}

\newtheorem{proposition}[theorem]{Proposition}

\theoremstyle{definition}
\newtheorem{remark}[theorem]{Remark}

\def\myffrac#1#2 in #3{\raise 2.6pt\hbox{$#3 #1$}\mkern-1.5mu\raise 0.8pt\hbox{$
		#3/$}\mkern-1.1mu\lower 1.5pt\hbox{$#3 #2$}}

\def\qed{\hfill $\vcenter{\hrule height .3mm
		\hbox {\vrule width .3mm height 2.1mm \kern 2mm \vrule width .3mm
			height 2.1mm} \hrule height .3mm}$ \bigskip}

\def \id {{\rm Id}}

\def \cov {{ \rm Cov}\,}
\def \var {{ \rm Var}\,}

\def \susbeteq {\subseteq}

\begin{document}

\title{Spectral monotonicity under Gaussian convolution}
\author{Bo'az Klartag and Eli Putterman}
\date{}
\maketitle

\abstract{We show  that the Poincar\'e constant of a log-concave measure in Euclidean space
is monotone increasing along the heat flow. In fact, the entire spectrum of the associated Laplace operator
is monotone decreasing. Two proofs of these results are given. The first proof analyzes a curvature
term of a certain time-dependent diffusion, and the second proof constructs a contracting transport map
following the approach of Kim and Milman.}

\section{Introduction}
\label{sec1}

The Poincar\'e constant $C_{P}(\mu)$ of a Borel probability measure $\mu$ on $\RR^n$ is the smallest constant
$C \geq 0$ such that for any locally-Lipschitz function $f \in L^2(\mu)$,
$$\var_{\mu}(f) \leq C \cdot \int_{\RR^n} |\nabla f|^2 d \mu $$
where $\var_{\mu}(f) = \int_{\RR^n} f^2 d \mu - \left( \int_{\RR^n} f d \mu \right)^2$ and $| \cdot |$ is the Euclidean norm.
The Poincar\'e constant
governs the rate of convergence to equilibrium of the Langevin dynamics in velocity space \cite{LG}.

\medskip
Suppose that $\mu$ admits a smooth, positive density $\rho$ in $\RR^n$.
The Laplace operator associated with $\mu$, defined a priori on smooth, compactly supported functions $u: \RR^n \rightarrow \RR$, is given by
\begin{equation}  L u = L_{\mu} u = \Delta u + \nabla (\log \rho) \cdot \nabla u. \label{eq_1030}\end{equation}
It satisfies
$$ \int_{\RR^n} (L u) v d \mu = -\int_{\RR^n} \langle \nabla u, \nabla v \rangle d \mu $$
for any two smooth functions $u,v: \RR^n \rightarrow \RR$, one of which is compactly supported. The operator $L_{\mu}$ is essentially self-adjoint
in $L^2(\mu)$, negative semi-definite, with a simple eigenvalue at $0$ corresponding to the constant eigenfunction (see \cite[Corollary 3.2.2]{BGL}).
The Poincar\'e constant is given by
$$ C_P(\mu) = 1 / \lambda_1^{(\mu)}, $$ where
$\lambda_1^{(\mu)}$ is the {\it spectral gap} of $L$, the infimum over all positive $\lambda > 0$ that belong to the spectrum of $-L$.
Under mild regularity assumptions the spectrum of $L$ is discrete (e.g., when $\rho$ is $C^2$ and $\Delta(\sqrt{\rho}) / \sqrt{\rho}$ tends to infinity at infinity
\cite[Corollary 4.10.9]{BGL}, or when $\rho$ is log-concave and $|\log \rho(x)| / |x|$ tends to infinity at infinity, as shown in Appendix \ref{app:log_conc_disc} below).
In this case we write
$$ 0 = \lambda_0^{(\mu)} < \lambda_1^{(\mu)} \leq \lambda_2^{(\mu)} \leq \lambda_3^{(\mu)} \leq \ldots $$
for the eigenvalues of $-L$, repeated according to their multiplicity.

\medskip A non-negative function $\rho$ on $\RR^n$ is log-concave if
$K = \{ x \in \RR^n \, ; \, \rho(x) > 0 \}$ is convex,
and $\log \rho$ is concave in $K$.  An absolutely continuous probability measure on $\RR^n$
is called log-concave if it has a log-concave density.
An arbitrary probability measure on $\RR^n$
is called log-concave if it is the pushforward of some absolutely continuous log-concave probability measure on $\RR^k$
under an injective affine map. An example of a log-concave probability measure is $\gamma_s$,
the Gaussian probability measure on $\RR^n$ of mean zero and covariance $s \cdot \id$.
In a minor abuse of notation, we use $\gamma_s$ to denote also its density function $\gamma_s(x) = (2 \pi s)^{-n/2} \exp(-|x|^2/(2 s))$.
Another example of a log-concave probability measure is the uniform
probability measure on any convex body in $\RR^n$. The convolution of two log-concave probability measures
is again log-concave, as follows from the Pr\'ekopa-Leindler inequality \cite[Theorem 1.2.3]{BGVV} or from the earlier work
by Davidovi\v{c}, Korenbljum and Hacet \cite{DKH}.

\medskip The Poincar\'e constant is a particularly useful invariant in the class of log-concave probability measures.
For example, when $\mu$ is absolutely-continuous and log-concave, its Poincar\'e constant
is determined, up to a multiplicative universal constant, by the isoperimetric constant
$$ h(\mu) = \inf_{A \subseteq \RR^n} \frac{\int_{\partial A} \rho}{\min \{ \mu(A), 1 - \mu(A) \}}, $$
where the infimum runs over all open sets $A \susbeteq \RR^n$ with smooth boundary. Indeed, the Cheeger \cite{Ch}
and Buser-Ledoux \cite{Buser,Led} inequalities state that for any absolutely-continuous, log-concave probability measure $\mu$ on $\RR^n$,
$$ \frac{1}{4} \leq C_P(\mu) \cdot h^2(\mu) \leq 9. $$
A well-known conjecture by Kannan, Lov\'asz and Simonovits (KLS) states
that the Poincar\'e constant of a log-concave probability measure is equivalent, up to a multiplicative universal constant,
to the operator norm of the covariance matrix of $\mu$. See the recent paper by Chen  \cite{C} for more background and for the best known result towards this conjecture.

\medskip
Abbreviate $\gamma = \gamma_1$, the standard Gaussian measure in $\RR^n$,
whose Poincar\'e constant is  $C_P(\gamma) = 1$ (e.g., \cite[Proposition 4.1.1]{BGL}). It was proven by
Cattiaux and Guillin \cite[Theorem 9.4.3]{CG} that when $\mu$ is a log-concave probability measure,
\begin{equation}  C_P(\mu) \leq C_P(\mu * \gamma) + 1, \label{eq_1151}
\end{equation}
where $\mu * \gamma$ is the convolution of $\mu$ and $\gamma$. The reverse inequality $C_P(\mu) \geq C_P(\mu * \gamma) - 1$
is much easier to obtain and does not require log-concavity (see, e.g., \cite[Proposition 1]{BBN}). Our main result in this paper is an improvement upon (\ref{eq_1151}):

\begin{theorem} Let $\mu$ be a log-concave probability measure on $\RR^n$. Then,
\begin{equation}  C_p(\mu) \leq C_P(\mu * \gamma). \label{eq_1125} \end{equation}
Moreover, assuming that $\mu$ admits a density that is smooth and positive in $\RR^n$ and that $L_{\mu}$ has a discrete spectrum, we have
$$ \lambda_k^{(\mu * \gamma)} \leq \lambda_k^{(\mu)} \qquad \qquad (k=1,2,\ldots) $$
\label{thm_1115}
\end{theorem}

Two proofs of Theorem \ref{thm_1115} are presented here. One of these proofs utilizes a method from Kim and Milman \cite{KM}
to construct a contraction transporting $\mu * \gamma$ to $\mu$. Recall
that a map $T: \RR^n \rightarrow \RR^n$ is a contraction if $|T(x) - T(y)| \leq |x-y|$ for all $x, y\in \RR^n$.

\begin{theorem} Let $\mu$ be a log-concave probability measure on $\RR^n$. Then there exists
a contraction $T: \RR^n \rightarrow \RR^n$ that pushes forward $\mu * \gamma$ to $\mu$.
\label{thm_154}
\end{theorem}
This result is reminiscent of Caffarelli's theorem \cite{Caf}, which states that there is a contraction
pushing forward $\gamma$ to $\mu$ in the case where the density of $\mu$ with respect to the measure $\gamma$ is log-concave.
As is well-known, Theorem \ref{thm_154} implies that the Poincar\'e constant of $\mu$ is not larger
than that of $\mu * \gamma$. Moreover, as explained e.g. in Ledoux \cite[Proposition 1.2]{Led1},
it follows from Theorem \ref{thm_154} that when $\mu$ is an absolutely-continuous,
log-concave probability measure on $\RR^n$,
\begin{equation}
h(\mu) \geq h(\mu * \gamma). \label{eq_223}
\end{equation}
There is also a corresponding
inequality between the log-Sobolev constants of $\mu$ and $\mu * \gamma$,
or any other quantity involving a Rayleigh-type quotient, see Caffarelli \cite[Corollary 8]{Caf}.
We explain the proof of Theorem \ref{thm_154} and its implications in \S  \ref{sec_KM}.

\medskip We continue with a discussion of an additional proof of Theorem \ref{thm_1115}, which
was chronologically the first proof that we found. For $s > 0$ denote
$$ \mu_s = \mu * \gamma_s, $$
the evolution of the measure $\mu$ under the heat flow.
The log-concavity of $\mu$ implies that $\mu_s$ is log-concave as well. We will show that $C_P(\mu_s)$
is nondecreasing in $s$.  For a function $f: \RR^n \rightarrow \RR$ we
consider its evolution under the heat semigroup
$$P_s(f) = f * \gamma_s \qquad \qquad \qquad (s > 0) $$
whenever the integrals defining the convolution converge absolutely. Setting $P_0 = \id$ we obtain the heat semigroup $(P_s)_{s \geq 0}$, which satisfies the Euclidean heat equation
\begin{equation}  \frac{\partial}{\partial s} P_s f  = \frac{\Delta P_s f}{2}   \qquad \qquad \qquad (s > 0). \label{eq_1028} \end{equation}
The operator
$P_s: L^2(\mu_s) \rightarrow L^2(\mu)$ is a contraction operator with $P_s(1) = 1$ since
$$ \| P_s(f) \|_{L^2(\mu)}^2 = \int_{\RR^n} P_s(f)^2 \rho \leq \int_{\RR^n} P_s(f^2) \rho = \int_{\RR^n} f^2 P_s(\rho) = \| f \|_{L^2(\mu_s)}^2. $$
Write $\rho_s = \rho * \gamma_s$ for the  density of the probability measure $\mu_s$, which is a smooth positive function in $\RR^n$.
The adjoint operator $Q_s = P_s^*: L^2(\mu) \rightarrow L^2(\mu_s)$ is defined by $Q_0 \vphi = \vphi$ and
\begin{align}
Q_s \vphi &= \frac{P_s(\vphi \rho)}{\rho_s} \qquad \qquad \qquad (s > 0), \label{eq_945} \end{align}
and again it is a contraction operator with $Q_s(1) = 1$. It follows from (\ref{eq_1028}) and (\ref{eq_945}) that the evolution
equation for $Q_s$ is the parabolic equation
\begin{equation}  \frac{\partial}{\partial s} Q_s \vphi  = \frac{\Delta Q_s \vphi}{2}   + \nabla \log \rho_s \cdot \nabla Q_s \vphi.
\label{eq_1231} \end{equation}
We are thus led to define the ``box operator''
\begin{equation}  \Box_s u = \frac{\Delta u}{2} + \nabla (\log \rho_s) \cdot \nabla u. \label{eq_234} \end{equation}
This operator resembles the Laplace operator $L_s := L_{\mu_s}$. Indeed, we have
\begin{equation}
L_s = \Box_s + \frac{\Delta}{2}. \label{eq_955}
\end{equation}
The $\Box_s$ operator obeys a Bochner-type formula, which is unsurprising as $\Box_s$
equals half of the Laplace operator associated with
the log-concave probability measure whose density is proportional to $\rho_s^2$.
Indeed, we compute that for smooth $u,v: \RR^n \rightarrow \RR$,
\begin{equation} \Gamma_2(u,v) := \Box_s \langle \nabla u, \nabla v \rangle - \langle \nabla \Box_s u, \nabla v \rangle -
\langle \nabla  u, \nabla \Box_s v \rangle \label{eq_446} \end{equation} satisfies
\begin{equation} \Gamma_2(u,u) = \| \nabla^2 u \|_{HS}^2 - 2 \nabla^2 (\log \rho_s) \nabla u \cdot \nabla u, \label{eq_447} \end{equation}
where $\| \nabla^2 u \|_{HS}$ is the Hilbert-Schmidt norm of the Hessian matrix $\nabla^2 u$.
The expression in (\ref{eq_447}) is similar to the Bochner-type formula of the operator $L_s$, the main difference
being the factor $2$ in front of the second summand in (\ref{eq_447}), which is the ``curvature term.''
Moreover, setting $\Gamma_0(u,v) = uv$ and $\Gamma_1(u,v) = \nabla u \cdot \nabla v$,
we have
$$ \frac{d}{ds} \int_{\RR^n} \Gamma_i(Q_s \vphi, Q_s \vphi) d \mu_s = -\int_{\RR^n} \Gamma_{i+1}(Q_s \vphi, Q_s \vphi) d \mu_s \qquad \qquad (i=0,1), $$
under some regularity assumptions to be explained below. It follows that the Rayleigh quotient
$$ \frac{\int_{\RR^n} |\nabla Q_s \vphi|^2 d \mu_s}{\int_{\RR^n} (Q_s \vphi)^2 d \mu_s } $$
is non-increasing in $s \in (0, \infty)$. This fact, formulated as Theorem \ref{cor_612} below,
implies  Theorem \ref{thm_1115}. More details, explanations and rigourous proofs
are provided in \S \ref{sec2}.

\medskip
In \S \ref{sec3} we discuss conceptual aspects of the evolution $(Q_s \vphi)_{s \geq 0}$, and explain how it is equivalent to
{\it Eldan's stochastic localization} \cite{eldan, LV} and {\it F\"ollmer's drift} \cite{lehec}. We also provide a
Bayesian interpretation of this evolution, and explore various connections between these points of view.

\medskip
{\it Acknowledgement.}  Supported by a grant from the Israel Science Foundation (ISF).

\section{A dynamic variant of $\Gamma$-calculus}
\label{sec2}

In this section we prove Theorem \ref{thm_1115}. Consider the linear differential operator $\Box_s$ defined
by formula (\ref{eq_234}) above.
Similarly to the formalism from \cite{BGL}, for smooth functions $u, v: \RR^n \rightarrow \RR$ we
define $\Gamma_0(u,v) = uv$, and for $i \geq 0$ and $s > 0$,
\begin{equation} \Gamma_{i+1}(u,v) = \Box_s \Gamma_i(u, v ) - \Gamma_i(u, \Box_s v ) - \Gamma_i(\Box_s u, v ) - \frac{d \Gamma_i}{ds} (u, v). \label{eq_212} \end{equation}
Thus $\Gamma_1(u,v) = \nabla u \cdot \nabla v$ and  $\Gamma_2(u,v)$ coincides with definition (\ref{eq_446}) above.
The rationale for  definition (\ref{eq_212}) is that, under regularity assumptions stated below,
\begin{equation} \frac{d}{ds} \int_{\RR^n} \Gamma_i(\vphi_s, \vphi_s) d \mu_s  = -\int_{\RR^n} \Gamma_{i+1}(\vphi_s, \vphi_s) d \mu_s.
\label{eq_216} \end{equation}
where $\vphi_s = Q_s \vphi$. If we were allowed to ignore all regularity issues, \eqref{eq_216} could be proven as follows: differentiating under the integral sign and applying (\ref{eq_1028}) and (\ref{eq_1231}),
\begin{align} \nonumber \frac{d}{ds} \int_{\RR^n} & \Gamma_i(\vphi_s, \vphi_s) \rho_s  =  2 \int_{\RR^n}
\Gamma_i \left( \Box_s \vphi_s, \vphi_s \right) \rho_s + \int_{\RR^n}
\Gamma_i \left( \vphi_s, \vphi_s \right) \frac{\Delta \rho_s}{2} + \int_{\RR^n} \frac{d \Gamma_i}{ds} (\vphi_s, \vphi_s) \rho_s \\ & =
 \int_{\RR^n}
\left[ 2 \Gamma_i \left( \Box_s \vphi_s, \vphi_s \right)  + \frac{\Delta \Gamma_i \left( \vphi_s, \vphi_s \right)}{2}
 \right] d \mu_s + \int_{\RR^n} \frac{d \Gamma_i}{ds} (\vphi_s, \vphi_s) d \mu_s. \label{eq_602}
 \end{align}
Next we use (\ref{eq_955}) and the fact that $\int_{\RR^n} (L_s u) d \mu_s = 0$ under regularity assumptions (e.g., when $u$ is smooth and compactly supported). This yields
\begin{align} \nonumber \frac{d}{ds} \int_{\RR^n} \Gamma_i(\vphi_s, \vphi_s) \rho_s & =
 \int_{\RR^n}
\left[ 2 \Gamma_i \left( \Box_s \vphi_s, \vphi_s \right)  - \Box_s
\Gamma_i \left( \vphi_s, \vphi_s \right) \right] d \mu_s + \int_{\RR^n} \frac{d \Gamma_i}{ds} (\vphi_s, \vphi_s) d \mu_s \\ & = -\int_{\RR^n} \Gamma_{i+1}(\vphi_s, \vphi_s) d \mu_s. \label{eq_603}
 \end{align}
This would be a rigorous proof for (\ref{eq_216}) had we worked in the context of a compact Riemannian
manifold (which also has a heat kernel $P_s: L^2(\mu) \rightarrow L^2(\mu_s)$ and a corresponding adjoint $Q_s = P_s^*$).
However, in this paper we are interested in the non-compact situation of $\RR^n$, since we rely on the
fact that the heat flow preserves curvature conditions such as log-concavity,
which is currently known to hold only for a Euclidean space \cite{Koles}. Nevertheless, the operators $\Box_s$ and $(Q_s)_{s \geq 0}$ seem rather natural also in the Riemannian setting.

\medskip Our first task in this section is to rigorously justify \eqref{eq_216} for a fairly large class of functions $\vphi$. To do this, we shall express $Q_s$ explicitly as an integral operator.

\medskip Recall that we work with an absolutely-continuous, log-concave probability measure $\mu$ on $\RR^n$ having density $\rho$.
  As before, for $s > 0$ we write $\mu_s = \mu * \gamma_s$ and $\rho_s = P_s \rho = \rho * \gamma_s$, while the operator $Q_s$
  is defined via formula (\ref{eq_945}).
For $s > 0$ and $y \in \RR^n$ we define the probability density
 \begin{equation}  p_{s, y}(x) = e^{\frac{\langle y, x \rangle}{s} -  \frac{|x|^2}{2s} } \frac{\rho(x)}{Z_{s,y}} \label{eq_350} \end{equation}
where
\begin{equation}  Z_{s,y} = \int_{\RR^n} e^{\frac{\langle y, x \rangle}{s} -  \frac{|x|^2}{2s} } \rho(x) dx = \frac{\rho_{s}(y)}{\gamma_{s}(y)} \label{eq_350_} \end{equation}
  is a normalizing constant (the ``partition function''). In the next lemma we express the value of $Q_s \vphi$ at the point $y$
  as the average of $\vphi$ with respect to the density $p_{s,y}$.

\begin{lemma} Let $y \in \RR^n, s >0$ and suppose that $\vphi \in L^1(\mu)$, or more generally,
that $\vphi: \RR^n \rightarrow \RR$ is such that $\vphi(x) e^{-t |x|^2} \in L^1(\mu)$  for some $t \in (0, s)$. Then
\begin{equation} Q_s \vphi(y) = \int_{\RR^n} \vphi \cdot p_{s, y}. \label{eq_405} \end{equation}
Furthermore, $Q_s \vphi(y)$ is a smooth function of $y \in \RR^n$ and $s > 0$ which satisfies
\begin{align} \nabla Q_s \vphi & = \frac{Q_s( x \vphi) - Q_s(\vphi) Q_s(x)}{s}, \label{eq_404}  \\
\partial_s Q_s \vphi(y) & = -\frac{Q_s( f_y \vphi) - Q_s(\vphi) Q_s(f_y) }{s^2} = \Box_s Q_s \vphi(y), \label{eq_432} \end{align}
where $f_y(x) = \langle x, y \rangle - |x|^2 / 2$.
\label{lem_606}
\end{lemma}

\begin{proof} According to (\ref{eq_350}) and (\ref{eq_350_}),
\begin{align} \rho_{s}( y ) \cdot \int_{\RR^n} \vphi \cdot p_{s, y}  \label{eq_958}
& = \int_{\RR^n} \vphi(x)  \cdot e^{\langle y, x \rangle / s-   |x|^2 / (2s)} \rho(x) dx \cdot \gamma_{s} (y) \\
 & =  \int_{\RR^n} \vphi(x) \rho(x) \cdot \gamma_s(y-x) dx = [(\vphi \rho) * \gamma_s](y) = P_s(\vphi \rho). \nonumber \end{align}
Now (\ref{eq_405}) follows from (\ref{eq_945}) and (\ref{eq_958}). The smoothness of
$Q_s \vphi$ and equations (\ref{eq_404}) and (\ref{eq_432})
follow by differentiating (\ref{eq_350_}) and (\ref{eq_405}) under the integral sign. This is legitimate, since
any partial derivative in the $(s,y)$-variables of the function $p_{s,y}(x) \vphi(x)$ is seen to be bounded by an integrable function,
and the bound is locally uniform in $s$ and $y$.
\end{proof}

By a multi-index $(k, \alpha)$ we mean a non-negative integer $k$ and a vector $\alpha = (\alpha_1,\ldots,\alpha_n)$ of nonnegative integers.
For a multi-index $(k,\alpha)$ and for a smooth function $f(s,y)$  we abbreviate
$$ \partial_s^{k} \partial_y^\alpha  f =  \left( \frac{\partial}{\partial s} \right)^{k}
\left( \frac{\partial}{\partial y_1} \right)^{\alpha_1}
\ldots  \left( \frac{\partial}{\partial y_n} \right)^{\alpha_n}  f(s,y) \qquad \qquad \qquad (s > 0, y \in \RR^n). $$
We denote $|\alpha| = \sum_i |\alpha_i|$. We say that a
measurable function $\vphi: \RR^n \rightarrow \RR$ has {\it subexponential decay relative to $\rho$} if for any $a > 0$ there exists $C > 0$ such that $|\vphi(x)| \le \frac{C}{\sqrt{ \rho(x)}} e^{-a |x|}$ for all $x \in \mathbb R^n$ for which $\rho(x) > 0$.

\begin{lemma}\label{rs_bds}
Fix $s > 0$ and let $(k, \alpha)$ be a multi-index. Then,
\begin{enumerate}
\item[(i)] The function
$\partial_s^{k} \partial_y^{\alpha} \log \rho_s(y)$ grows at most polynomially at infinity in $y \in \RR^n$.
\item[(ii)]
Let
$\vphi$ have subexponential decay relative to $\rho$. Then the
function $\partial_s^{k} \partial_y^{\alpha}  Q_s \vphi(y)$ has  subexponential decay relative to $\rho_s$.
\end{enumerate}
Moreover, if $s$ varies in an interval $[s_0, s_1]$ with $s_0 > 0$, then the implied constants in these two assertions
may be chosen not to depend on $s$.
\end{lemma}

\begin{proof} Using the heat equation, we can replace time derivatives of $\rho_s$ by space derivatives, so in (i) we only need consider space derivatives of $\log \rho_s$. By differentiating (\ref{eq_350_})  with respect to $y$ we see that
$$ \nabla_y \log \rho_s(y)  = -\frac{y}{s} + \nabla_y \log Z_{s,y}  = \frac{Q_s(x)(y) - y}{s}.  $$
Repeated differentiations show that conclusion (i) would follow once we prove the following claim: for any $d > 0$, the function $Q_s(|x|^d)(y)$ grows at most polynomially at infinity as a function of $y \in \RR^n$,
with the implied constants not depending on $s \in [s_0, s_1]$.

\medskip Let us prove this claim. Since $\rho$ is an integrable, log-concave function, there exist $a, b > 0$ such that
\begin{equation} \rho(x) \leq a e^{-b|x|} \qquad \qquad \text{for all} \ x \in \RR^n. \label{eq_208} \end{equation}
 (see e.g., \cite[Lemma 2.2.1]{BGVV}). In particular $C_d := \int_{\RR^n} (1+|x|)^d \rho(x) dx < \infty$. Therefore, for any $y \in \RR^n$ and $s > 0$,
$$ \frac{\int_{\RR^n} |x-y|^d e^{-|x-y|^2 / (2s)} \rho(x) dx }{\int_{\RR^n} e^{-|x-y|^2 / (2s)} \rho(x) dx}
\leq \frac{\int_{\RR^n} |x-y|^d \rho(x) dx }{\int_{\RR^n} \rho(x) dx} \leq C_d (1 + |y|)^d, $$
where the first inequality follows
from the fact that $|x-y|^d$ is increasing in $|x-y|$ while $e^{-|x-y|^2/(2s)}$ is decreasing in $|x-y|$, and the second inequality uses $|x-y| \leq (1+|x|) (1 + |y|)$.
Consequently, as $|x|^d \leq 2^d( |y|^d + |x-y|^d )$,
\begin{align*}
Q_s(|x|^d)(y) = \int_{\RR^n} |x|^d p_{s,y}(x) dx = \frac{\int_{\RR^n} |x|^d e^{-|x-y|^2 / (2s)} \rho(x) dx }{\int_{\RR^n} e^{-|x-y|^2 / (2s)} \rho(x) dx} \leq \tilde{C}_d (1 + |y|)^{d},
\end{align*}
for some coefficient $\tilde{C}_d$ depending only on $\rho$ and on $d$. This shows that $Q_s(|x|^d)(y)$ grows at most polynomially, from which (i) follows.

\medskip We move on to the proof of (ii). Given $a > 0$, let $C > 0 $ be such that $\sqrt{\rho(x)} \cdot |\vphi(x)| \le C e^{-a |x|}$
for all $x$. From (\ref{eq_945}),
\begin{equation}|\vphi_s| \le \frac{C}{P_s(\rho)} P_s( \sqrt{\rho(x)} e^{-a |x|}) \le \frac{C}{P_s(\rho)} \cdot \sqrt{ P_s(\rho)  P_s(e^{-2a|x|}) } = \frac{C}{\sqrt{\rho_s}} P_s(e^{-2a|x|})^{\frac{1}{2}} \label{eq_404_}
\end{equation}
where we have used the Cauchy-Schwarz inequality for $P_s$. In order to conclude that $\vphi_s$ has
subexponential decay relative to $\rho_s$, it remains only to note the following: since $P_s$ is convolution with a Gaussian of covariance $s \cdot \id$, there exists $\tilde{C} = \tilde{C}_{a,s,n} > 0$ such that $P_s(e^{-2 a|x|})(y) \le \tilde{C} e^{-2a |y|}$ for all $y \in \RR^n$. (The constant $\tilde{C} = \sup_{s \in [s_0, s_1]} \int_{\RR^n} e^{2 a |x|} \gamma_s(x) dx$ works for all $s \in [s_0, s_1]$.)

\medskip We still need to bound the partial derivatives of $\vphi_s(y)$ with respect to the $s$-variable and $y$-variables.
The first-order derivatives are given by formulas (\ref{eq_404}) and (\ref{eq_432}), and higher-order
derivatives may be computed by repeated applications of these two formulas. Thus
$\partial_s^{k} \partial_y^{\alpha}  Q_s \vphi(y)$ can be expressed as a sum with a fixed number of summands.
Each of these summands is a product of a term of the form $\frac{1}{s^m} Q_s(f \vphi)$, where $f$ is a polynomial of degree bounded by $2k + |\alpha|$, and terms of the form $Q_s(p)$ with $p$ a polynomial in the space variables. For any such $p$,
the function $Q_s(p)$ grows at most polynomially because $Q_s(|x|^d)$ does for all $d$. In addition, $f \vphi$ has subexponential decay relative to $\rho$, so by the previous part of the proof, $Q_s(f \vphi)$ has subexponential decay relative to $\rho_s$. Consequently, each of the summands in $\partial_s^{k} \partial_y^{\alpha}  Q_s \vphi(y)$ has subexponential decay relative to $\rho_s$, so $\partial_s^{k} \partial_y^{\alpha}  Q_s \vphi(y)$ does as well.
\end{proof}

Recall the definition (\ref{eq_212}) of $\Gamma_i(u,v)$. In the next proposition we rigorously justify the computations in (\ref{eq_602}) and (\ref{eq_603}). We discuss only the case $i=0,1$; while the extension to higher-order carr\'es des champs presents no particular difficulty, it has been omitted as it is unnecessary for our purposes.

\begin{proposition} Fix $s > 0$ and $i=0,1$. Suppose that $\vphi$ has subexponential decay relative to $\rho$ and $\vphi_s = Q_s(\vphi)$. Then
equation (\ref{eq_216}) holds for $\vphi_s$.
 \label{prop_542}
\end{proposition}

\begin{proof} Recall that $\Gamma_0(u,v) = uv, \, \Gamma_1(u,v) = \nabla u \cdot \nabla v$
and that $\vphi_s(y)$ is smooth in $(s,y)$.
Therefore,
\begin{equation} \partial_s \left[ \Gamma_i(\vphi_s, \vphi_s) \rho_s \right]
= 2 \Gamma_i (\partial_s \vphi_s, \vphi_s) \rho_s + \Gamma_i(\vphi_s, \vphi_s) \rho_s \cdot \partial_s(\log \rho_s). \label{eq_446_} \end{equation}
According to Lemma \ref{rs_bds} we may bound the expression in (\ref{eq_446_})
by the integrable function $C e^{-a |y|}$ for some $C, a > 0$, and the bound is locally uniform in $s$. This justifies interchanging differentiation and integration to obtain
\begin{align*} \partial_s \int_{\RR^n} \Gamma_i(\vphi_s, \vphi_s) \rho_s &= \int_{\RR^n} \partial_s \left[ \Gamma_i(\vphi_s, \vphi_s) \rho_s \right]  \\
&= 2 \int_{\RR^n}
\Gamma_i \left( \Box_s \vphi_s, \vphi_s \right) \rho_s + \int_{\RR^n}
\Gamma_i \left( \vphi_s, \vphi_s \right) \frac{\Delta \rho_s}{2},
\end{align*}
where we have used Lemma \ref{lem_606} and the heat equation (\ref{eq_1028}). Next we need to carry out the integrations by parts of (\ref{eq_602}) and (\ref{eq_603}) and show that no boundary terms arise. When integrating the term $\Gamma_i(\vphi_s, \vphi_s) \Delta \rho_s /2$ by parts twice, we encounter the boundary integrands $\Gamma_i(\vphi_s, \vphi_s)  \nabla \rho_s$ and $\nabla \Gamma_i(\vphi_s, \vphi_s) \cdot \rho_s$. Both of these decay exponentially at infinity, so the integration by parts over $\mathbb R^n$ introduces no boundary terms, verifying (\ref{eq_602}). In (\ref{eq_603}), we use the integration by parts formula
$$ \int_{\RR^n} L_s \Gamma_i(\vphi_s, \vphi_s) \cdot \rho_s = \int_{\RR^n} {\rm div}( \rho_s \nabla \Gamma_i(\vphi_s, \vphi_s) ) = 0, $$
which is again justified by the exponential decay of $\nabla \Gamma_i(\vphi_s, \vphi_s) \cdot \rho_s$ at infinity. This completes the proof
of (\ref{eq_216}).
\end{proof}

We write $H^1(\mu)$ for the space of all functions in $L^2(\mu)$ whose weak derivatives belong to $L^2(\mu)$, equipped with the norm
$$ \| f \|_{H^1(\mu)} = \sqrt{ \int_{\RR^n}  f^2 d \mu + \int_{\RR^n}  |\nabla f|^2 d \mu}. $$
See e.g. the appendix of \cite{BK} and the references therein for information about weak derivatives, the Sobolev space $H^1(\mu)$,
and for a proof of the fact that the space of smooth, compactly supported functions in $\RR^n$ is dense in $H^1(\mu)$.

\begin{theorem} Let $\mu$ be an absolutely-continuous, log-concave probability measure on $\RR^n$ and let $0 \not \equiv \vphi \in H^1(\mu)$.  Then with $\vphi_s = Q_s \vphi$, the Rayleigh quotient
\begin{equation} R_{\vphi}(s) = \frac{\int_{\RR^n} |\nabla \vphi_s |^2 d \mu_s}{\int_{\RR^n} \vphi_s^2 d \mu_s}
\label{eq_406} \end{equation}
is non-increasing in $s \in [0, \infty)$. Consequently, the function
 $\log \| \vphi_s \|_{L^2(\mu_s)}$ is convex in $s \in [0,\infty)$.
  \label{cor_612}
\end{theorem}

For the proof of Theorem \ref{cor_612} we require the following technical lemma:

\begin{lemma}\mbox{}
\begin{enumerate}
\item[(i)] For any fixed $s \geq 0$, the quantities  $R_{\vphi}(s)$ and $\| \vphi_s \|_{L^2(\mu_s)}$ depend continuously on  $\vphi \in H^1(\mu) \setminus \{ 0 \}$.
\item[(ii)] For any $0 \not \equiv \vphi \in H^1(\mu)$,
\begin{equation}  R_{\vphi}(0) = \lim_{s \rightarrow 0^+} R_{\vphi}(s) \qquad \text{and} \qquad \| \vphi \|_{L^2(\mu)} = \lim_{s \rightarrow 0^+} \| \vphi_s \|_{L^2(\mu_s)}. \label{eq_1313} \end{equation}
\end{enumerate}
\label{lem_1321}
\end{lemma}

\begin{proof} We first prove part (i). Let $\vphi \in L^2(\mu) \backslash \{0\}$. By Cauchy-Schwarz, we have the pointwise bound
\begin{equation}  \vphi_s^2 \rho_s = \frac{P_s^2(\rho \vphi)}{P_s(\rho)} \leq \frac{P_s( \vphi^2 \rho) P_s(\rho)}{P_s(\rho)}= P_s(\vphi^2 \rho).
\label{eq_237} \end{equation}
Moreover, the log-concavity of $\mu$ implies that whenever $\vphi \in H^1(\mu)$,
\begin{equation}
|\nabla \vphi_s|^2 \rho_s \leq P_s(|\nabla \vphi|^2 \rho).
\label{eq_254}
\end{equation}
Indeed, the probability density $p_{s,y}$ from (\ref{eq_350}) is ``more log-concave than $\gamma_s$,'' in the sense
that $p_{s,y} / \gamma_s$ is log-concave. The Brascamp-Lieb inequality (see e.g., \cite[\S 4.9]{BGL})
thus implies that the Poincar\'e constant of the probability density $p_{s,y}$ is at most $s$. That is,
letting $X$ be a random vector with density $p_{s,y}$ and $f$ a weakly differentiable function with $\EE |f(X)|^2 < \infty$ and $\EE |\nabla f(X)|^2 < \infty$,
\begin{equation}  \var f(X) \leq s \cdot \EE |\nabla f(X)|^2. \label{eq_1256} \end{equation}
Hence, by Lemma \ref{lem_606}, for any  $\theta \in S^{n-1} = \{ x \in \RR^n \, ; \, |x| = 1 \}$,
$$
\nabla \vphi_s \cdot \theta = \frac{\EE (X \cdot \theta) \vphi(X) - \EE (X \cdot \theta) \EE \vphi(X)}{s}
\leq \frac{\sqrt{ \var(X \cdot \theta) \var(\vphi(X)) } }{s} \leq \sqrt{ \EE |\nabla \vphi(X)|^2 },
$$
which implies (\ref{eq_254}) since $\EE |\nabla \vphi(X)|^2 = Q_s(|\nabla \vphi|^2) = P_s(|\nabla \vphi|^2 \rho) / \rho_s$.
By integrating over $\RR^n$, the inequalities (\ref{eq_237}) and (\ref{eq_254}) imply that
\begin{align*}
& \| \vphi_s \|_{L^2(\mu_s)}  \leq \| \vphi \|_{L^2(\mu)} \qquad \qquad & \text{for} \ \vphi \in L^2(\mu), \\
& \| \nabla \vphi_s \|_{L^2(\mu_s)}  = \sqrt{ \int_{\RR^n} |\nabla \vphi_s|^2 d \mu_s } \leq \| \vphi \|_{H^1(\mu)} \qquad \qquad & \text{for} \ \vphi \in H^1(\mu).
\end{align*}
Consequently, the functional $\vphi \mapsto \| \vphi_s \|_{L^2(\mu_s)}$ is $1$-Lipschitz in $L^2(\mu)$,
while the functional $\vphi \mapsto \| \nabla \vphi_s \|_{L^2(\mu_s)}$ is $1$-Lipschitz in $H^1(\mu)$.
In particular, for any fixed $s \geq 0$, the quantities  $R_{\vphi}(s)$ and $\| \vphi_s \|_{L^2(\mu_s)}$ depend continuously on $\vphi \in H^1(\mu) \setminus \{ 0 \}$, proving (i).

\medskip For part (ii), note that $\vphi \mapsto R_{\vphi}(s)$ is locally uniformly continuous in $H^1(\mu) \setminus \{ 0 \}$, being the quotient of two positive, $1$-Lipschitz functions.
Hence, it suffices to prove (\ref{eq_1313}) for $\vphi$ in a dense subset of $H^1(\mu) \setminus \{ 0 \}$. We may thus assume that $\vphi$ is smooth and compactly supported. We claim that
for almost every $y \in \RR^n$,
\begin{equation}  \vphi_s^2(y) \rho_s(y) \xrightarrow{s \rightarrow 0}  \vphi^2(y) \rho(y)
\qquad \text{and} \qquad  |\nabla \vphi_s(y)|^2 \rho_s(y) \xrightarrow{s \rightarrow \infty}  |\nabla \vphi(y)|^2 \rho(y).
\label{eq_305} \end{equation}
Let $K = \{ x \in \RR^n \, ; \, \rho(x) > 0 \}$. In proving (\ref{eq_305}), we may thus assume that $y \not \in \partial K$, since the boundary of
the convex set $K$ has  Lebesgue measure zero.
If $y \not \in \overline{K}$ then $\rho$ vanishes in a neighborhood of $y$, hence
$$ P_s(\vphi^2 \rho)(y) \xrightarrow{s \rightarrow 0} 0  \qquad \text{and} \qquad  P_s(|\nabla \vphi|^2 \rho)(y) \xrightarrow{s \rightarrow 0} 0, $$
and (\ref{eq_305}) follows from the bounds (\ref{eq_237}) and (\ref{eq_254}). As $\rho$ is log-concave, it is locally Lipschitz on $K$, so by the Rademacher theorem, $\rho$ is differentiable almost everywhere in the interior of $K$. It thus suffices to prove (\ref{eq_305}) for $y \in K$ such that $\rho$ is differentiable at $y$. Differentiating $\vphi_s$ yields
\begin{equation}\label{nabla_phis2}
\nabla \vphi_s = \nabla\left(\frac{P_s(\vphi\rho)}{P_s(\rho)}\right) = \frac{\nabla P_s(\vphi\rho)}{P_s(\rho)} - \frac{P_s(\vphi\rho) \nabla \rho_s}{\rho_s^2} \qquad \qquad \qquad (s > 0).
\end{equation}
It is a property of the heat semigroup that if $f$ is a bounded measurable function differentiable at a point $y \in \RR^n$, then $\nabla P_s(f)(y) \rightarrow \nabla f(y)$
as $s \rightarrow 0$; this is easily shown by writing $\nabla P_s f = f * \nabla \gamma_s$ and approximating $f$ by its first-order Taylor polynomial. Applying this to the functions $\rho$ and $\vphi \rho$ which are bounded in $\RR^n$ and differentiable at $y$, we obtain $\nabla \rho_s(y) \rightarrow \nabla\rho(y)$ and $\nabla P_s(\vphi \rho)(y) \rightarrow \nabla(\vphi \rho)(y)$ as $s \rightarrow 0$. Moreover,
$\vphi_s(y) \rightarrow \vphi(y), \rho_s(y) \rightarrow \rho(y)$ because $\rho$ and $\vphi\rho$ are continuous at $y$ and bounded in $\RR^n$.
It follows that $\nabla\vphi_s(y) \rightarrow \nabla \vphi(y)$, completing the proof  of (\ref{eq_305}).

\medskip Finally, the functions $|\nabla \vphi|^2 \rho$ and $\vphi^2 \rho$ are bounded and compactly supported.
Hence, for $s \in (0, 1]$ the Gaussian convolutions $P_s(|\nabla\vphi|^2 \rho)$ and $P_s(\vphi^2 \rho)$
are bounded by $C e^{-|y|^2 / 2}$ in $\RR^n$ for some constant $C$ that does not depend on $s$.
From (\ref{eq_237}), (\ref{eq_254}), the dominated convergence theorem, and (\ref{eq_305}), we obtain
$$ \int_{\RR^n} \vphi_s^2 \rho_s \xrightarrow{s \rightarrow 0}  \int_{\RR^n} \vphi^2 \rho
\qquad \text{and} \qquad  \int_{\RR^n} |\nabla \vphi_s|^2 \rho_s \xrightarrow{s \rightarrow 0}  |\nabla \vphi|^2 \rho, $$
completing the proof of (\ref{eq_1313}).
\end{proof}

\begin{remark} Inequality (\ref{eq_254}) states that $|\nabla \vphi_s|^2 \leq Q_s( |\nabla \vphi|^2 )$.
Using the interpretation in \S 3.1 and arguing as in \cite[\S 3.2]{Stroock}, one may prove the stronger gradient bound $$ |\nabla \vphi_s| \leq Q_s(|\nabla \vphi|), $$ which we do not need here.
\end{remark}

It follows from (\ref{eq_212}) and a straightforward computation that
\begin{equation}  \Gamma_2(u,u) = \Box_s |\nabla u|^2 - 2 \langle \nabla \Box_s u, \nabla u \rangle
= \| \nabla^2 u \|_{HS}^2 - 2 \langle \nabla^2 (\log \rho_s) \nabla u, \nabla u \rangle.
\label{eq_535} \end{equation}
On the other hand, the Bochner formula for the differential operator $L_s = L_{\mu_s}$ states that for any smooth, compactly supported
function $u: \RR^n \rightarrow\RR$,
\begin{equation}
\int_{\RR^n} (-L_s)^2 u \cdot u d \mu_s = \int_{\RR^n} (L_s u)^2 d \mu_s = \int_{\RR^n} \left[ \| \nabla^2 u \|_{HS}^2 - \langle \nabla^2 (\log \rho_s) \nabla u, \nabla u \rangle
\right] d \mu_s. \label{eq_536}
\end{equation}
See \cite[\S 1.16.1]{BGL} for a proof of (\ref{eq_536}). Formula (\ref{eq_536}) remains
valid when $u$ and its partial derivatives are smooth functions with subexponential decay relative to $\rho_s$,
since the integration by parts yield no boundary terms as in the proof of Proposition \ref{prop_542}. Thanks
to Lemma \ref{rs_bds}, we know that formula (\ref{eq_536}) is valid for $u = Q_s \vphi$ whenever $\vphi$ has subexponential decay relative to $\rho$.

\medskip
The integrand on the right-hand side of (\ref{eq_536}) is almost identical to the expression in (\ref{eq_535}),
the only difference is the coefficient $2$ in front of the second summand.

\begin{proof}[Proof of Theorem \ref{cor_612}]
When $u$ is a smooth function such that $u$ and its partial derivatives have subexponential decay relative to $\rho_s$, we write for $i=1,2$,
$$ \| u \|_{\dot{H}^i(\mu_s)} = \sqrt{ \int_{\RR^n} (-L_s)^i u \cdot u d \mu }. $$
Thus  $\| u \|^2_{\dot{H}^1(\mu_s)} = \int_{\RR^n} |\nabla u|^2 d\mu_s$.
The operator $L_s$ is initially defined by the formula $L_s u = \Delta u + \nabla \log \rho_s \cdot \nabla u$
assuming $u$ and its partial derivatives have subexponential decay relative to $\rho_s$.
This operator is essentially self-adjoint and negative semi-definite
in $L^2(\mu_s)$ (e.g., \cite[Corollary 3.2.2]{BGL}). Hence, by the spectral theorem and the Cauchy-Schwarz inequality,
\begin{equation}
\| u \|_{\dot{H}^1(\mu_s)}^2 \leq \| u \|_{\dot{H}^2(\mu_s)} \cdot \| u \|_{L^2(\mu_s)}. \label{eq_552}
\end{equation}
Consider first the case where  $0 \not \equiv \vphi \in H^1(\mu)$ has subexponential decay relative to $\rho$
and $s > 0$.
Thanks to Proposition \ref{prop_542} we may apply
(\ref{eq_216}) and compute  that
$$ \frac{d}{ds} R_{\vphi}(s) = \frac{\| \vphi_s \|_{\dot{H}^1(\mu_s)}^4 - \int_{\RR^n} \Gamma_2(\vphi_s, \vphi_s) d \mu_s \cdot \| \vphi_s \|_{L^2(\mu_s)}^2  }{ \| \vphi_s \|_{L^2(\mu_s)}^4 }. $$
Hence, from (\ref{eq_535}) and (\ref{eq_536}),
$$ \label{rayleigh_diff}
\frac{d}{ds} R_{\vphi}(s)  = \frac{\|\vphi_s\|^4_{\dot{H}^1(\mu_s)} - \|\vphi_s\|^2_{\dot{H}^2(\mu_s)} \|\vphi_s\|^2_{L^2(\mu_s)}}{\|\vphi_s\|^4_{L^2(\mu_s)}} + \frac{\int_{\RR^n} \langle (\nabla^2 \log \rho_s)\nabla \vphi_s, \nabla \vphi_s\rangle d\mu_s}{\|\vphi_s\|^2_{L^2(\mu_s)}}.
$$
By log-concavity $\nabla^2 \log \rho_s \leq 0$. Hence
we conclude from (\ref{eq_552}) that
$$ \frac{d}{ds} R_{\vphi}(s) \leq 0.  $$
Therefore $R_{\vphi}(s)$ is non-increasing in $s \in (0, \infty)$. It follows from (\ref{eq_216}) that $$ \partial_s \log \| \vphi \|_{L^2(\mu_s)}
= -R_{\vphi}(s), $$ and consequently $\log \| \vphi \|_{L^2(\mu_s)}$ is convex in $s \in (0, \infty)$.
Lemma \ref{lem_1321} now implies that $R_{\vphi}(s)$ is decreasing in $s \in [0, \infty)$
and $\log \| \vphi_s \|_{L^2(\mu_s)}$ is convex in $s \in [0, \infty)$.

\medskip
Finally, compactly supported smooth functions, which certainly have subexponential decay relative to $\rho$, are dense in $H^1(\mu)$.
The Rayleigh quotient and $\| \vphi_s \|_{L^2(\mu_s)}$ are continuous on $H^1(\mu) \backslash \{0\}$ by Lemma \ref{lem_1321}, hence
we obtain that $R_\vphi(s)$ is non-increasing and
$\log \| \vphi_s \|_{L^2(\mu_s)}$ is convex in $s \in [0, \infty)$ for any $0 \not \equiv \vphi \in H^1(\mu)$.
\end{proof}

Using the min-max characterization of eigenvalues, we derive our main result as a corollary to Theorem \ref{cor_612}.

\begin{proof}[Proof of Theorem \ref{thm_1115}] We may set $s = 1$, since $\mu_s = \mu * \gamma$ for $s =1$.
We may assume that $\mu$ is absolutely continuous, as otherwise we may pass
to a lower dimension thanks to the well-known fact that the Poincar\'e constant of a Cartesian
product of two measures is the maximum of the Poincar\'e constants of the factors.
The Poincar\'e constant of $\mu$, which is finite and positive (see \cite{bobkov}), satisfies
$$
\frac{1}{C_P(\mu)} = \inf \left\{R_\vphi(0) \, ;\, 0 \not \equiv \vphi \in H^1(\mu), \int_{\RR^n} \vphi\,d\mu = 0\right\},
$$
and similarly for $\mu_s$. For any $\eps > 0$ there exists $0 \not \equiv \vphi \in H^1(\mu)$ with $\int \vphi d \mu = 0$
such that $R_\vphi(0) < C_P(\mu)^{-1} + \eps$.
Since $\int \vphi_s d\mu_s = \int \vphi d \mu = 0$, we deduce from Theorem \ref{cor_612} that,
$$ \frac{1}{C_P(\mu_s)} \leq R_{\vphi}(s) \leq R_{\vphi}(0) < \frac{1}{C_P(\mu)} + \eps. $$
As $\eps > 0$ was arbitrary, inequality (\ref{eq_1125}) is proven.

\medskip
Next, assume that $L_{\mu}$
has discrete spectrum, and let $k \geq 1$. There exists a $(k+1)$-dimensional subspace $E \subseteq H^1(\mu)$
such that $R_{\vphi}(0) \leq \lambda_k^{(\mu)}$ for any $0 \not \equiv \vphi \in E$. For $s > 0$ the linear operator $Q_s$
defined in (\ref{eq_945}) is one-to-one in $L^1(\mu)$. (Indeed, given $P_s(\vphi \rho)$ we may recover the Fourier transform of $\vphi \rho \in L^1(\RR^n)$ which determines $\vphi \in L^1(\mu)$.) Hence
$$ E_s = \{ Q_s \vphi \, ; \, \vphi \in E \} $$
is a $(k+1)$-dimensional subspace, and $R_\vphi(s) \leq R_{\vphi}(0) \le \lambda_k^{(\mu)}$ for all $\vphi \in E$. In other words, there exists a $(k+1)$-dimensional
subspace $E_s \subseteq H^1(\mu_s)$ on which the Rayleigh quotient is at most $\lambda_k^{(\mu)}$. By the min-max characterization of eigenvalues,
$$ \lambda_k^{(\mu_s)} \leq \lambda_k^{(\mu)}, $$
completing the proof.
\end{proof}

The proof of Theorem \ref{thm_1115} clearly shows that $C_P(\mu * \gamma_s) \geq C_P(\mu)$
for all $s > 0$, so by the semigroup property $s \mapsto C_P(\mu * \gamma_s)$ is non-decreasing in $s \in [0, \infty)$.

\begin{remark}
Let $\mu$ be a log-concave probability measure in $\RR^n$ with density $\rho = e^{-W}$, where $W$ is a smooth function
such that $$ \lim_{x \rightarrow \infty} \frac{|\nabla W(x)|^2}{2} - \Delta W(x) = \infty. $$
In this case, we have the strict inequality
\begin{equation}  \lambda_k^{(\mu * \gamma)} < \lambda_k^{(\mu)} \qquad \qquad (k=1,2,\ldots) \label{eq_433} \end{equation}
In order to  prove (\ref{eq_433}), we first observe
that $\nabla^2 \log \rho_s(y) < 0$
for all $y \in \RR^n$ as follows from
the equality case of the Brascamp-Lieb inequality
or from \cite{CF, DPF}. Arguing as in the
proof of Theorem \ref{cor_612} and using the fact that $\nabla \vphi_s \not \equiv 0$ as $\vphi_s$
is non-constant, we conclude that $d R_{\vphi}(s) / ds < 0$ whenever
$0 \not \equiv \vphi \in H^1(\mu)$ has subexponential decay relative to $\rho$.

\medskip Therefore (\ref{eq_433}) would follow from Theorem \ref{cor_612}, as in the proof of Theorem \ref{thm_1115} above, had
we known that any eigenfunction $\vphi$ of $L_{\mu}$ has subexponential decay relative to $\rho$.

\medskip Indeed, let $A: L^2(\RR^n) \overset{\sim}\rightarrow L^2(\mu)$ be the isometry given by $A(g) = e^{\frac{W}{2}} g$. It is well-known and easy to verify that $A^{-1} L_\mu A$ is the Schr\"odinger operator
$$-\Delta + \frac{|\nabla W|^2}{4} - \frac{\Delta W}{2},$$
which is of the form $-\Delta + V$ with $V \ge 0$ and $V \rightarrow \infty$ as $x \rightarrow \infty$. By results on the decay of eigenfunctions of Schr\"odinger operators \cite[Theorem XIII.70]{RS}, the function  $A^{-1} \vphi$ has subexponential decay at infinity, and hence $\vphi$ has subexponential decay relative to $\rho$.
\end{remark}

\section{A contraction transporting $\mu * \gamma$ to $\mu$}
\label{sec_KM}

In this section we prove Theorem \ref{thm_154} using the arguments of Kim and Milman \cite{KM}.
To begin with, we work with a log-concave probability measure $\mu$ with a smooth, strictly positive density $\rho$ on $\mathbb R^n$.
We furthermore make the regularity assumption
that there exists $\eps > 0$ such that
\begin{equation} -\nabla^2 \log \rho(y) \leq \frac{1}{\eps} \cdot \id \qquad \qquad (y \in \RR^n). \label{eq_520}
\end{equation}
We shall later remove these assumptions on $\rho$. As above, for $s \geq 0$ we write $\mu_s = \mu * \gamma_s$ and $\rho_s$ is the density of $\mu_s$. Thus $\rho_s$ is smooth, positive
and log-concave in $\RR^n$. For $s \geq 0$ consider the {\it advection field}
\begin{equation} W_s(y) = -\frac{1}{2} \nabla \log \rho_s(y). \label{eq_443} \end{equation}
The ``physical'' interpretation of this vector field is as follows. One of the derivations of the heat equation is based on Fourier's law, according to
which the flux of heat across a tiny surface in a short time interval is proportional to the temperature gradient
across the surface. If we think of the heat as carried by a fluid of particles with density $\rho(x, t)$, this means that the current of heat is proportional to $-\nabla \rho$ (we take $\frac{1}{2}$ to be the constant of proportionality); since the current of heat is simply $\rho v$, where $v(x, t)$ is the bulk velocity of the fluid, we obtain $v = -\frac{1}{2} \frac{\nabla \rho}{\rho}$, which is \eqref{eq_443}. For more details, see \cite[\S 5.4]{villani}.

With this point of view, the trajectory of a particle located at time $s = 0$
at the point $y \in \RR^n$ is the curve $s \mapsto T_s(y)$ where
\begin{equation} \left \{  \begin{array}{rcll} \frac{d}{ds} T_s(y)  & = & W_s(T_s(y)), & s \geq 0 \\ T_0(y) & = & y & \end{array} \right.
\label{eq_452} \end{equation}

\begin{lemma} Under the above assumptions on $\rho$, the ordinary differential equation (\ref{eq_452}) determines
the family of maps $(T_s: \RR^n \rightarrow \RR^n)_{s \geq 0}$. These maps are
all diffeomorphisms, and $T_s(y)$ is smooth in $(s,y) \in [0, \infty) \times \RR^n$.
\end{lemma}

\begin{proof} Since $\rho_s = \rho * \gamma_s$, the function $\rho_s(y)$ is smooth and positive in $(s,y) \in [0, \infty) \times \RR^n$.
Therefore $W_s(y)$ is smooth in $(s,y) \in [0, \infty) \times \RR^n$ as well.
It remains to show that $W_s$ is $1/(2\eps)$-Lipschitz on $\RR^n$ for any $s \geq 0$. Once this is shown,
 the standard theory of ordinary differential equations implies the existence and uniqueness
of solutions to \eqref{eq_452} and their smooth dependence on initial conditions (e.g., \cite[Chapter V]{Hartman}). The fact that the $T_s$ are diffeomorphisms follows from the theory of flows of time-dependent vector fields on manifolds (e.g., \cite[Chapter 17]{Lee}).

 \medskip We need to compute the derivative of $W_s$. As in the beginning of the proof of Lemma \ref{rs_bds} above,
 by differentiating (\ref{eq_350_}) we see that for any $s > 0$ and $y \in \RR^n$,
\begin{equation}  D W_s(y) = -\frac{1}{2} \nabla^2 \log \rho_s(y) = \frac{s \cdot \id - \cov(p_{s,y})}{2 s^2} \label{eq_717} \end{equation}
where $\cov(p_{s,y}) \in \RR^{n \times n}$ is the covariance matrix of the probability density $p_{s,y}$.
Since $\rho_s$ is log-concave, the differential $D W_s$ is a symmetric positive semidefinite matrix.
From (\ref{eq_350}) and the regularity assumption (\ref{eq_520}) we see  that for $s \geq 0$ and $x \in \RR^n$,
\begin{equation}
 -\nabla^2 \log p_{s,y}(x) \leq \left( \frac{1}{\eps} + \frac{1}{s} \right) \cdot \id
\label{eq_607}
\end{equation}
in the sense of symmetric matrices.  It is well-known (see \cite[Theorem 5.4]{BLL}) that (\ref{eq_607}) implies that
\begin{equation}
\left( \frac{1}{\eps} + \frac{1}{s} \right)^{-1}  \cdot \id \leq \cov( p_{s,y} ).
\label{eq_608}
\end{equation}
 From (\ref{eq_717}) and (\ref{eq_608}) we deduce the pointwise bound
$$ \left \| D W_s(y) \right \|_{op} \leq \frac{1}{2(s+\eps)} \qquad \qquad \qquad (y \in \RR^n), $$
where $\| \cdot \|_{op}$ is the operator norm.
This bound clearly applies also for $s = 0$.
Therefore $W_s: \RR^n \rightarrow \RR^n$ is $1/(2\eps)$-Lipschitz for any $s \geq 0$, completing the proof.
\end{proof}

As explained in Kim and Milman \cite{KM}, the diffeomorphism $T_s$ is an expansion, i.e., $|T_s(x) - T_s(y)| \geq |x-y|$
for all $x,y$ and $s$. In order to prove this, we show that everywhere in $\RR^n$,
\begin{equation}  (D T_s)^* (D T_s) \geq \id.
\label{eq_306} \end{equation}
Inequality (\ref{eq_306}) is certainly true when $s = 0$, while the fact that $D W_s$ is positive semidefinite implies that
$$ \frac{\partial}{\partial s} (D T_s)^* (D T_s)
= 2 (D T_s)^* (D W_s) D T_s \geq 0. $$
Therefore (\ref{eq_306}) holds true. This implies that $\left \| D \left(T_s^{-1} \right) \right \|_{op} \leq 1$,
hence $T_s^{-1}$ is a contraction and $T_s$ is an expansion. Next, from (\ref{eq_443}) and the heat
equation $\partial \rho_s / \partial s = \Delta \rho_s / 2$ we obtain the linear transport equation
(also known as the continuity equation),
$$ \frac{\partial \rho_s}{\partial s} + {\rm div}( \rho_s W_s ) = 0 \qquad \qquad \qquad (s \geq 0, y \in \RR^n). $$
The continuity equation implies that $\rho_s$ is the density of the pushforward of $\mu$ under
the diffeomorphism $T_s$ (see e.g. \cite[Theorem 5.34]{villani}).
Consequently, the map $T_s^{-1}$ is a contraction that pushes forward $\mu_s$ to $\mu$.

\begin{proof}[Proof of Theorem \ref{thm_154}] Set $s = 1$ so that $\mu_s = \mu * \gamma$.
We have just established the existence of a contraction transporting $\mu_s$ to $\mu$
under the additional requirement that $\mu$ admits a smooth, positive density satisfying
the regularity assumption (\ref{eq_520}).

\medskip Consider now the case where $\mu$ is an arbitrary absolutely-continuous, log-concave probability measure
in $\RR^n$. For any $\eps > 0$, the measure $\mu_{\eps} = \mu * \gamma_{\eps}$ has a smooth, positive, log-concave density
satisfying the regularity assumption (\ref{eq_520}), as follows from the computation in (\ref{eq_717}) above.
Hence there exists a contraction transporting $\mu_{\eps} * \gamma$ to $\mu_{\eps}$.
By \cite[Lemma 3.3]{KM}, in order to show that there exists a contraction from $\mu * \gamma$ to $\mu$, it suffices to show that
\begin{equation}  \mu_{\eps} \xrightarrow{\eps \rightarrow 0^+} \mu \label{eq_336} \end{equation}
in the total variation metric, and that $\mu_{\eps} * \gamma \longrightarrow \mu * \gamma$ as $\eps \rightarrow 0$ in the weak topology.
Since $\mu_{\eps} * \gamma = (\mu * \gamma)_{\eps}$ and since convergence in total variation
implies convergence in the weak topology, it suffices to prove (\ref{eq_336}). Thus we need to show that
\begin{equation} \int_{\RR^n} |\rho_{\eps}(x) - \rho(x)| dx \xrightarrow{\eps \rightarrow 0^+} 0. \label{eq_406_} \end{equation}
Arguing as in (\ref{eq_208}) and the paragraph following (\ref{eq_404_}) above, we know that there exist $a,b > 0$ such that
 $\rho_{\eps}(x) \leq a e^{-b|x|}$ for all $x \in \RR^n$ and $0 \leq \eps \leq 1$, with $\rho_0 = \rho$.
Since $\rho$ is continuous almost everywhere in $\RR^n$, the integrand in (\ref{eq_406_})
converges to zero almost everywhere, and (\ref{eq_406_}) follows from the dominated convergence theorem.

\medskip Thus the conclusion of the theorem is valid when $\mu$ is an absolutely continuous, log-concave probability measure.
Finally, if $\mu$ is not absolutely continuous, then we may project to a lower
dimension using an orthogonal projection, which is a contraction, and reduce matters
to the absolutely continuous case.
\end{proof}

Theorem \ref{thm_154} implies that for $\nu = \mu * \gamma$ and  $0 \not \equiv \vphi \in H^1(\nu)$
we have the following inequality between Rayleigh quotients:
\begin{equation}
\frac{\int_{\RR^n} |\nabla (\vphi \circ T)|^2 d \nu}{\int_{\RR^n} (\vphi \circ T)^2 d \nu} \leq
\frac{\int_{\RR^n} |(\nabla \vphi) \circ T|^2 d \nu}{\int_{\RR^n} (\vphi \circ T)^2 d \nu} =
\frac{\int_{\RR^n} |\nabla \vphi |^2 d \mu}{\int_{\RR^n} \vphi^2 d \mu}. \label{eq_414} \end{equation}
We may now repeat the proof of Theorem \ref{thm_1115} from \S \ref{sec2}, with the linear
map $\vphi \mapsto \vphi \circ T$ playing the role of the linear map $\vphi \mapsto Q_1 \vphi$.
This yields another proof of Theorem \ref{thm_1115}, relying on (\ref{eq_414})
in place of Theorem \ref{cor_612}.

\section{A Bayesian interpretation of Eldan's stochastic localization}
\label{sec3}

Eldan's stochastic localization technique was introduced by Eldan in \cite{eldan} and developed since then by several authors in different settings
\cite{C, EG, K, LV}. The method has turned out to be useful in particular for the study of log-concave measures, culminating thus far in the breakthrough result of Chen \cite{C} showing that the isotropic constant grows more slowly than any power of the dimension. In this section, we give a ``Bayesian'' interpretation of Eldan's stochastic localization relating it to the heat flow and to the operator $Q_s$ introduced above, as well as to the F\"ollmer drift in the theory of Wiener space.
It was this line of development which led us to the results announced in the introduction; however, this section 
may be read independently.

\medskip
We refer to \cite{Oks} for background on stochastic processes. Let $\mu$ be an absolutely continuous probability measure on $\RR^n$ with density $p_0$ and with finite second moments. Let $(W_t)_{t \ge 0}$ be a standard Brownian motion on $\RR^n$ with $W_0 = 0$.

\medskip The stochastic localization process, in the version introduced by \cite{LV}, is a density-valued stochastic process $p_t$ driven by $W_t$, defined as follows: for every $x \in \RR^n$, the process $(p_t(x))_{t \geq 0}$ is the solution to the stochastic differential equation
\begin{equation}
dp_t(x) = p_t(x) \langle x - a_t, dW_t\rangle \label{eq_0051}
\end{equation}
with initial condition $p_0$, where $a_t = \int_{\RR^n} x \cdot p_t(x)\,dx$ is the barycenter of $p_t$. As this equation has no drift term, $p_t(x)$ is a martingale, and $p_t$ is almost surely a probability density. In particular, $\EE[p_t(x)] = p_0(x)$, and for any test function $\vphi$, we have $\EE_{X \sim p_0}[\vphi(X)] = \EE[\EE_{X \sim p_t}[\vphi(X)]]$.

\medskip
The process $(p_t)_{t \geq 0}$ has another description, as a stochastic ``tilt'' of $p_0$. In this section, for $t \ge 0$ and  $\theta \in \RR^n$ let $p_{t, \theta}$ denote the probability density given by
\begin{equation}\label{eq_0334}
p_{t, \theta}(x) = \frac{1}{Z(t, \theta)} e^{\langle \theta, x\rangle - \frac{t|x|^2}{2}} p_0(x),
\end{equation}
where $Z(t, \theta) = \int_{\RR^n} e^{\langle \theta, x\rangle - \frac{t|x|^2}{2}} p_0(x)\,dx$ is a normalization constant.
Let $a(t, \theta)$ denote the barycenter of $p_{t, \theta}$, and define the stochastic process $\theta_t$ via the differential equation
\begin{equation}\label{eq_0141}
d\theta_t = a(t, \theta_t)\,dt + dW_t, \qquad \qquad \qquad \theta_0 = 0.
\end{equation}
It turns out that when $\theta_t$ and $p_t$ are driven by the same Brownian motion, $p_t$ is precisely equal to $p_{t, \theta_t}$. For proofs of these and other formulas relating to the stochastic localization process, and for the application to the KLS conjecture, see \cite{LV, LV2} or \cite{C}.

\medskip The Bayesian interpretation of the Eldan process is quite simple: let $X$ be a random vector distributed according to $\mu$,
independent of the Brownian motion $(W_t)_{t \geq 0}$. Denote
\begin{equation}  \tilde{\theta}_t = t X + W_t \qquad \qquad (t \geq 0). \label{eq_1127} \end{equation}
Our main observations are the following two claims:
\begin{enumerate}
\item[(i)] The process $(\tilde{\theta}_t)_{t \geq 0}$ coincides in law with the process $(\theta_t)_{t \geq 0}$ which solves (\ref{eq_0141}) above.
\item[(ii)] For any fixed $t > 0$ and $\theta \in \RR^n$, the probability density $p_{t, \theta}$ on $\RR^n$ is precisely the conditional probability distribution of $X$ given that $\tilde{\theta}_t = \theta$.
\end{enumerate}

Thus, when we observe the tilt process $(\theta_t)_{t \geq 0}$, we actually see a Brownian motion with a constant drift $X$ which is unknown, but whose prior distribution is known to us. Moreover, the posterior probability density for the unknown drift $X$ given the observation of the process $(\theta_s)_{0 \leq s \leq t}$ until time $t$ depends only on the state of the process at time $t$, and is equal to $p_{t, \theta_t}$.

\medskip In the following proposition we prove these two claims. For $T > 0$ let $\cV_T = C_0([0, T], \RR^n)$ be the Wiener space of $\RR^n$-valued continuous functions $(W_t)_{0 \leq t \leq T}$
with $W_0 = 0$. Slightly abusing notation, we write $\gamma_T$ for the Wiener probability measure on $\cV_T$ and $\{\mathcal F_t\}_{0 \leq t \leq T}$ for the natural filtration,
i.e., $\cF_t$ is the $\sigma$-algebra generated by  $(W_s)_{0 \leq s \leq t}$.

\begin{proposition} Let $\mu$ be a probability measure on $\RR^n$ which is absolutely continuous with respect to the Lebesgue measure $\lambda$.
Fix $T > 0$, and consider the space $\Omega = \RR^n \times \cV_T$ and the transformation $\tau: \Omega \rightarrow \Omega$ given by
$$ \tau(x, (W_t)_{0 \leq t \leq T}) = (x, (W_t + tx)_{0 \leq t \leq T}). $$
Write $\nu = \tau_*(\mu \otimes \gamma_T)$. Then,
\begin{enumerate}
 \item[(i)] The stochastic process $(\tilde{\theta}_t)_{t \geq 0}$ described in (\ref{eq_1127}) coincides in law with the It\^o process $(\theta_t)_{t \geq 0}$ defined as the  solution to the stochastic differential equation (\ref{eq_0141}).
\item[(ii)] The measure $\nu$ is absolutely continuous with respect to $\lambda \otimes \gamma_T$ on $\Omega$ with density
$$ \frac{ d \nu }{d (\lambda \otimes \gamma_T) } ( x, \tilde{\theta}  )
 = p_0(x) e^{\langle \tilde{\theta}_T, x \rangle - \frac{T |x|^2}{2}}, $$
 for $x \in \RR^n$ and $\tilde{\theta} = (\tilde{\theta}_t)_{0 \leq t \leq T} \in \cV_T$.
Consequently, when $(X, (\tilde{\theta}_t)_{0 \leq t \leq T})$ is the stochastic process described in (\ref{eq_1127}),
 the conditional distribution of $X$ with respect to $\tilde{\theta} = (\tilde{\theta}_t)_{0 \leq t \leq T}$ is given by the
 probability density
 \begin{equation}  q_T(x | \tilde{\theta} ) = \frac{p_0(x) e^{\langle \tilde{\theta}_T, x\rangle - \frac{T |x|^2}{2}}}{\int_{\RR^n} p_0(y) e^{\langle \tilde{\theta}_T, y\rangle - \frac{T |y|^2}{2}} dy} = p_{T, \tilde{\theta}_T}(x) \qquad \qquad (x \in \RR^n). \label{eq_142} \end{equation}
\end{enumerate}
\end{proposition}

\begin{proof} We first prove (ii). For $x \in \RR^n$, let $\tau_x: \cV_T \rightarrow \cV_T$ be defined by $\tau_x((W_t)_{t \le T}) = (W_t + tx)_{t \le T}$
so that $\tau(x, \omega) = (x, \tau_x(\omega))$. By Fubini's theorem,
\begin{equation} \nu = \tau_*(\mu \otimes \gamma_T) = \int_{\RR^n} (x,\tau_x)_* \gamma_T  d \mu(x). \label{eq_1104} \end{equation}
This  means that for any test function $g$,
$$\int_{\Omega} g\,d\nu = \int_{\RR^n} \left( \int_{\cV_T} g(x, \tilde{\theta})\, d((\tau_x)_*\gamma_T)(\tilde{\theta}) \right) d \mu(x). $$
Since $\tau_x$ is just a translation in Wiener space by the deterministic function $f_x(t) = tx$,
the Cameron-Martin theorem \cite{CM} yields that the density of $(\tau_x)_* \gamma_T$ with respect to $\gamma_T$
at the point $(\tilde{\theta}_t)_{0 \leq t \leq T} \in \cV_T$ equals
\begin{align}
\frac{d(\tau_x)_* \gamma_T}{d\gamma_T}((\tilde{\theta}_t)_{t\le T}) &= \exp\left(\int_0^T \langle f_x'(t), d \tilde{\theta}_t \rangle - \frac{1}{2} \int_0^T |f_x'(t)|^2\,d t\right) \nonumber \\
&= \exp\left(\int_0^T x \,d \tilde{\theta}_t - \frac{1}{2} \int_0^T |x|^2\,d t\right) = e^{\langle \tilde{\theta}_T, x\rangle - \frac{T |x|^2}{2}}. \label{eq_1103}
\end{align}
It follows from (\ref{eq_1104}) and (\ref{eq_1103}) that
\begin{equation}
\frac{ d \nu }{d (\lambda \otimes \gamma_T) } ( x, \tilde{\theta} ) = p_0(x) \cdot \frac{ d \nu }{d (\mu \otimes \gamma_T) } ( x, \tilde{\theta} ) = p_0(x) e^{\langle \tilde{\theta}_T, x \rangle - T \frac{|x|^2}{2}}.
\label{eq_354} \end{equation}
The probability measure $\nu$ is the joint distribution of the
stochastic process  $$ (X, (\tilde{\theta}_t)_{0 \leq t \leq T}) $$ described in (\ref{eq_1127}).
Therefore, when conditioning on the entire stochastic process $(\tilde{\theta}_t)_{0 \leq t \leq T}$,
it follows from (\ref{eq_354}) that the probability density function of $X$ is proportional to $x \mapsto p_0(x) e^{\langle \tilde{\theta}_T, x \rangle - T \frac{|x|^2}{2}}$
in $\RR^n$.
This completes the proof of (ii).

\medskip We move on to the proof of (i). We endow $\Omega$ with the probability measure $\mu \otimes \gamma_T$,
and assume that $(X, (W_t)_{t \geq 0})$ is distributed according to this measure, while $\tilde{\theta}_t = t X + W_t$. Thus,
\begin{equation} d \tilde{\theta}_t = X dt + d W_t. \label{eq_355} \end{equation}
Write $\cN_t$ for the $\sigma$-algebra generated by $(\tilde{\theta}_s)_{0 \leq s \leq t}$.
Abbreviate $\EE[X | \tilde{\theta}] = \EE[X | \cN_t ](\tilde{\theta})$
for the conditional expectation of $X$ with respect to $\cN_t$, which is a function
of $(\tilde{\theta}_s)_{0 \leq s \leq t}$.
According to (\ref{eq_355}) and \cite[Theorem 8.4.3]{Oks}, the process $(\tilde{\theta}_t)_{0 \leq t \leq T}$ coincides in law with the process $(\theta_t)_{0 \leq t \leq T}$ defined by the initial condition $\theta_0 = \tilde{\theta}_0 = 0$ and the stochastic differential equation
$$
d \theta_t = b(t, \theta_t) \, d t + dW_t,
$$
if the function $b(t, x)$ defined for $0 < t \leq T$ and $x \in \RR^n$ satisfies
\begin{equation}  b(t, \tilde{\theta}_t) = \EE[X | \tilde{\theta} ]  \qquad \forall \tilde{\theta} \in \cV_t.
\label{eq_155} \end{equation}
(To be precise, the statement in \cite[Theorem 8.4.3]{Oks} only treats time-independent diffusions, but
the proofs generalize almost verbatim to the time-dependent case which we need.)
The random variable $\EE[X | \tilde{\theta}]$, viewed as an $\cN_t$-measurable function on $\Omega$,
is the conditional expectation of $X$ given $\tilde{\theta} = (\tilde{\theta}_s)_{0 \leq s \leq t}$. According to (ii),
the conditional distribution of $X$ given $\tilde{\theta}$ is given by the probability density $q_t(x | \tilde{\theta})$
 from (\ref{eq_142}). Hence for any $0 < t < T$ and $\tilde{\theta} \in \cV_t$,
$$ \EE[X | \tilde{\theta}] = \int_{\RR^n} x \cdot q_t(x | \tilde{\theta} )\,dx = \int_{\RR^n} x \cdot p_{t, \tilde{\theta}_t}(x)\,dx = a(t, \tilde{\theta}_t).
$$ We have thus verified condition (\ref{eq_155}) with $b(t,x) = a(t,x)$, completing the proof of (i).
\end{proof}

To reiterate, we have interpreted Eldan's stochastic localization for the measure $\mu$ as the following procedure: a value $x$ is sampled from the distribution $\mu$, and a Brownian motion with constant drift $x$, namely $\theta_t = tx + W_t$, is shown to an observer who knows the distribution $\mu$, but not the value of $x$. From the observer's perspective, $\theta_t$ satisfies the stochastic differential equation (\ref{eq_0141}), and at time $t$, the observer's posterior probability distribution for the hidden drift coefficient $x$ is precisely $p_{t, \theta_t}$. The fact that $p_{t, \theta_t}$ is a martingale now follows immediately from the law of total probability: for $s < t$,
\begin{align}
\EE[p_{t, \theta_t}(x) | \theta_s] &= \EE[p(X = x | \theta_t) | \theta_s] = \int_{\RR^n} p(X = x | \theta_t = \theta) p(\theta_t = \theta | \theta_s)\,d\theta \nonumber \\
&= \int_{\RR^n} p(X = x, \theta_t = \theta | \theta_s)\,d\theta = p(X = x | \theta_s) = p_{s, \theta_s}(x).
\end{align}
In Bayesian terms, this simply means that if we continually obtain information about an unknown random variable $X$ and update our posteriors for $X$ accordingly, our expectation at time $s$ for our estimate of $X$, or any function of $X$, at time $t$ must coincide with our current estimate of $X$.

\begin{remark}\mbox{}
\begin{enumerate}
\item[(i)] A curious property of Eldan's stochastic localization, in the Bayesian interpretation, is that the posterior distribution at time $t$ depends only on $\theta_t$: the full path $(\theta_s)_{0 \leq s \le t}$ contains no more information about $X$ than $\theta_t$ alone. This is a limiting case of an amusing exercise in linear algebra and statistics which we now describe. Suppose that we are given $N$ noisy observations of an unknown quantity $X$,  of the form
    \begin{equation}  X + Z_1, X + \frac{Z_1 + Z_2}{2}, \ldots, X + \frac{Z_1 + \ldots + Z_N}{N} \label{eq_1109} \end{equation}
    where $Z_1,\ldots, Z_N$ are independent, standard Gaussian random variables. Assume that the apriori distribution of $X$ is known to us.
    What is the posterior distribution of $X$ given the $N$ observations in (\ref{eq_1109})? As it turns out, the posterior distribution depends only on the last of these $N$ observations, for which the Gaussian noise is of the smallest variance. The first $N-1$ observations are completely useless in this context.
\item[(ii)] A suitably generalized version of this interpretation applies to the general stochastic localization process with a control matrix $C_t = C(t, \theta_t)$, as defined by \cite{LV2}: a random variable $X$ is drawn from $\mu$ as above, but instead of a Brownian motion with drift $X$, what the observer sees is an It\^o process defined by the SDE $d\theta_t = C(t, \theta_t) X\,dt + C(t, \theta_t)^{\frac{1}{2}}\,dW_t$. Again, $p_t$ represents the observer's posterior distribution for $X$ given the observation of $\theta_t$ up to time $t$. To prove that this description corresponds to the definition of the process in \cite{LV2} one repeats the above argument using Girsanov's theorem, rather than the Cameron-Martin theorem. The ``path-independence'' property of the posteriors from the previous remark does not hold in this case.
\end{enumerate}
\end{remark}

\subsection{Time inversion}

Let us now explain the relationship between the tilt process in its Bayesian interpretation and our work in \S 2.
A well-known identity for Brownian motion is the time-inversion property: suppose that  $(W_t)_{t \in [0, \infty)}$ is a standard Brownian motion in $\RR^n$ with $W_0 = 0$. Define $(\tilde W_s)_{s \in [0, \infty)}$ by $$ \tilde W_s = s W_{\frac{1}{s}} $$
and $\tilde W_0 = 0$. Then $(\tilde W_s)_{s \in [0, \infty)}$ is a standard Brownian motion as well.
Consequently, from (\ref{eq_1127}) we see that
\begin{equation}  Y_s := s \tilde{\theta}_{1/s} = X + \tilde{W}_s. \label{eq_1157} \end{equation}
Recalling that the tilt process $(\theta_t)_{t \geq 0}$ coincides in law with $(\tilde{\theta}_t)_{t \geq 0}$,
we conclude from (\ref{eq_1157}) that the tilt process coincides in law with the time inversion of a Brownian motion with a starting point drawn from the distribution $\mu$.

\medskip Applying this time inversion, we treat the time-inverted tilt process $(Y_s)_{s \geq 0}$ as
just a Brownian motion with a random starting point. Working with it requires nothing more than the explicit expression for the Euclidean heat kernel; for instance, the distribution of $Y_s$ is given by the probability density function $\rho_s = P_s \rho$ with $\rho = p_0$.
Given a function $\vphi$ on $\RR^n$ and $t > 0$, the random variable $$ \int_{\RR^n} \vphi p_t $$ associated to the tilt process coincides in law
with the distribution of $Q_s \vphi$ under the measure $\mu_s$, for $s = 1/t$. Moreover,
\begin{equation}  Q_{s} \vphi(y) = \EE \left[ \vphi(X) | \theta_t = \theta\right] \qquad \qquad \text{for} \ s = 1/t > 0, \theta = ty \in \RR^n. \label{eq_1124} \end{equation}
It is this elementary, ``functional analytic'' perspective on the measures $p_{t, \theta}$ -- or, in the new variables, $p_{s, y}$ -- that is taken in \S \ref{sec2}, which makes no explicit use of stochastic localization, pathwise analysis, martingales or stochastic calculus at all.

\subsection{F\"ollmer drift as a ``time-compressed'' version of stochastic localization}

F\"ollmer drift is a well-known stochastic process which couples between an absolutely continuous measure $\mu$ and Wiener measure on path space over a \textit{finite} time interval, without loss of generality $[0, 1]$. It is the same process
referred to as the ``$h$-process'' in Cattiaux and Guillin \cite{CG}, because of its relation to Doob's $h$-transform.

\medskip In brief, the F\"ollmer drift associated to $\mu$ is a Brownian motion conditioned to have law $\mu$ at time $t = 1$.
The measure $\cP^\mu$ on $\cV_1 = C_0[0, 1]$ defining the F\"ollmer drift of $\mu$ is defined as the measure having Radon-Nikodym derivative
\begin{equation} \frac{d \cP^\mu}{\gamma_1}(W) = \frac{d\mu}{d\gamma}(W_1) \qquad \qquad \qquad (W = (W_t)_{0 \leq t \leq 1} \in \cV_1), \label{eq_253} \end{equation}
where $\gamma_1$ on the left side of (\ref{eq_253}) is the Wiener measure on $C_0[0,1]$, while $\gamma$ on the right side
of (\ref{eq_253}) is the standard Gaussian measure in $\RR^n$. The F\"ollmer drift $\cP^\mu$ turns out to have a certain energy-minimizing property, and its energy is precisely twice the relative entropy $H(\mu | \gamma)$, properties which make it quite useful for proving functional inequalities; see, e.g., \cite{ELS, lehec}.

\medskip We can interpret F\"ollmer drift in a manner completely analogous to the tilt process: let $X$ be a random variable drawn from $\mu$ and let $(B_t)_{0 \leq t \leq 1}$ denote an independent standard Brownian bridge on $[0, 1]$. Then the law of the process
\begin{equation}  X_t = tX + B_t \qquad \qquad (0 \leq t \leq 1) \label{eq_259} \end{equation}
is precisely that of the F\"ollmer drift associated to $\mu$.
Moreover, just as above, one may consider an observer who sees $X_t$ but not $X$ and define posterior probability distributions for $X$ given $(X_s)_{s \le t}$. These posterior probability distributions are the random measures $\mu_t$ in \S 3 of \cite{ELS}, in a slightly different normalization (in \cite{ELS} the measure $\mu_t$ is the posterior probability distribution of $(X - X_t) / \sqrt{1-t}$ given $X_t$, rather
than the posterior probability distribution of $X$ itself given $X_t$).

\medskip
In fact, there is an even closer relationship between F\"ollmer drift $X_t$ and the tilt process $\theta_t$ of Eldan's stochastic localization, which manifests in two separate ways. First of all, for $t \le 1$, we may write
\begin{equation}  \theta_t = tX + W_t = t(X + W_1) + (W_t - t W_1). \label{eq_312} \end{equation} Note that $W_1$ is independent of $W_t - t W_1$ as these are
jointly Gaussian, centered and $\EE W_1 (W_t - t W_1) = 0$.
Recall that $B_t = W_t - tW_1$  is one way to construct a Brownian bridge. From (\ref{eq_259}) and (\ref{eq_312})
we see that the tilt process $(\theta_t)_{t \in [0, 1]}$ has the law of the F\"ollmer drift for the measure associated to $X + W_1$, namely $\mu * \gamma$. In the same fashion, one sees that for any $T > 0$, the process $(\theta_t)_{t \in [0, T]}$ is identical in law to the F\"ollmer drift of $\mu * \gamma_T$, with the time interval rescaled to $[0, T]$.

\medskip Another way to construct a standard Brownian bridge from a standard Brownian motion is by ``time compression'': the process
$$ \tilde{B}_t = (1 - t) W_{\frac{t}{1 - t}} \qquad \qquad (0 \leq t \leq 1) $$
coincides in law with the standard Brownian bridge. By inverting this operation, we can construct a Brownian motion from a Brownian bridge: $W_t = (1 + t) \tilde{B}_{\frac{t}{1 + t}}$.
Hence, the F\"ollmer drift for $\mu$ and the tilt process of Eldan's stochastic localization satisfy the reciprocal relations
$$ X_t \simeq (1 - t) \theta_{\frac{t}{1 - t}} \qquad \text{and} \qquad \theta_t \simeq  (1 + t) X_{\frac{t}{1 + t}}, $$
where $\simeq$ means ``coincides in law''.
This follows from the defining formulas $\theta_t = tX + W_t$, $X_t = tX + B_t$ and the corresponding relations for $B_t$ and $W_t$. Thus, F\"ollmer drift is simply a time-compressed version of the tilt process.

\appendix

\section{Discreteness of the spectrum for rapidly decreasing log-concave densities}\label{app:log_conc_disc}

The goal of this appendix is to prove the following proposition:

\begin{proposition}\label{prop:log_conc_disc} Let $\mu$ be a log-concave probability measure on $\mathbb R^n$ with smooth, positive density $\rho = e^{-V}$ such that $\frac{V(x)}{|x|} \rightarrow \infty$ as $x \rightarrow \infty$. Then the spectrum of $L_\mu$ is discrete.
\end{proposition}

For $A \subseteq \RR^n$ we write $C_c^{\infty}(A)$ for the class of smooth, compactly supported functions in $\RR^n$ that are supported in the set $A$.
As explained in \cite[\S 4.10]{BGL}, in order to prove Proposition \ref{app:log_conc_disc} it suffices to show the following:

\begin{enumerate}
\item[(*)] For any $a > 0$ there exists $r > 0$ such that for any $f \in C_c^{\infty}(\RR^n \setminus B_r)$,
$$ \int_{\RR^n} |\nabla f|^2 d \mu \geq a \cdot \int_{\RR^n} f^2 d \mu. $$
Here $B_r = \{ x \in \RR^n \, ; \, |x| \leq r \}$.
\end{enumerate}

Consider first the one-dimensional case in which $\mu$ is supported on a half-line. Thus $d \mu = \rho(x) dx =e^{-W(x)} dx$ is a measure on $[0, \infty)$
with $W: [0, \infty) \rightarrow \RR$ smooth and convex.
We will apply the Muckenhoupt criterion (\cite{M}; see also \cite[\S 4.5.1]{BGL}), which we state
as the following lemma:

\begin{lemma} Let $\rho: [0, \infty) \rightarrow (0, \infty)$ be such that $C := \sup_{r > 0} \int_r^\infty \rho \int_0^r \frac{1}{\rho} < \infty$. Then for every smooth, compactly supported function $f: [0, \infty) \rightarrow \RR$,
$$  \int_0^{\infty} f^2\,\rho \le  4C \int_0^{\infty} (f')^2\,\rho. $$
\end{lemma}

Let $x_0 > 0$ be such that $a := W'(x_0) > 0$. Then $W'(x) \ge a$ for all $x > x_0$ by convexity. Hence for any $r, x > x_0$,
\begin{align*}\rho(r) &\le e^{-a(r - x)} \rho(x) \quad \text{for $x < r$,} \\
\rho(x) &\le e^{-a(x - r)} \rho(r) \quad \text{for $x > r$.}
\end{align*}
Therefore, for any $r > x_0$,
\begin{align}
\int_{x_0}^r \frac{1}{\rho(x)}\,dx &\le \frac{1}{\rho(r)} \int_{x_0}^r e^{-a(r - x)}\,dx \le \frac{1}{\rho(r)} \cdot \frac{1}{a}, \\
\int_r^\infty \rho(x)\,dx &\le \rho(r) \int_r^\infty e^{-a(x - r)}\,dx \le \rho(r) \cdot \frac{1}{a}.
\end{align}
Thus we obtain $\sup_{r > x_0} \int_r^\infty \rho \int_{x_0}^r \frac{1}{\rho} \le \frac{1}{a^2}$, so Muckenhoupt's criterion yields that for any
$f \in C_c^{\infty}([x_0, \infty))$,
\begin{equation}  \int_{x_0}^{\infty} (f')^2 \,d\mu \ge \frac{a^2}{4} \int_{x_0}^{\infty} f^2 \,d\mu, \label{eq_449} \end{equation}
whenever $a = W'(x_0) > 0$.

\medskip So much for the one-dimensional case. Now let $d\mu = e^{-V}\,dx$ be an $n$-dimensional log-concave measure, and consider the family of functions $f_R: S^{n - 1} \rightarrow \mathbb R$ defined by $f_R(u) = \frac{V(Ru) - V(0)}{R}$. By convexity, $f_R$ is monotone increasing in $R$, and by assumption $f_R$ converges pointwise to infinity. Hence, applying Dini's theorem, we see that $f_R$ converges uniformly to $\infty$.
Denote $V_u(r) = V( r u )$ for $u \in S^{n-1}$ and $r \geq 0$. By convexity, $$ V_u'(r) \geq f_R(u). $$
Hence for every $a > 0$ there exists $\tilde{R} > 0$ such that $V_u'(r) \ge a$ for all $r \ge \tilde{R}$ and $u \in S^{n-1}$.
Denoting $W_u(r) = V_u(r) - (n-1) \log r$, we see
we see that $W_u$ is convex in $(0, \infty)$
and that for any $a > 0$ there exists $R = R(a) > 0$ such that $W_u'(r) \geq a$ for all $r \geq R$ and $u \in S^{n-1}$.
By integrating in polar coordinates and using (\ref{eq_449}) we conclude  that for any $a > 0$ and $f \in C_c^\infty(\mathbb R^n \setminus  B_R)$,
denoting $f_u(r) = f(ru)$,
\begin{align*}
\|\nabla f\|^2_{L^2(\mu)} & \geq \int_{S^{n-1}} \int_0^{\infty} r^{n-1} (f_u'(r))^2 e^{-V_u(r)} dr d u
 \geq \int_{S^{n-1}} \left( \int_{R}^{\infty} (f_u'(r))^2 e^{-W_u(r)} dr \right) d u \\ &
 \geq \frac{a^2}{4} \int_{S^{n-1}} \left( \int_R^{\infty} f_u(r)^2 e^{-W_u(r)} dr \right) d u = \frac{a^2}{4} \|f\|^2_{L^2(\mu)}.
\end{align*}
We have thus verified condition (*) above, completing the proof of Proposition \ref{prop:log_conc_disc}.

\medskip
\noindent Department of Mathematics, Weizmann Institute of Science, Rehovot 76100, Israel. \\
 {\it e-mails:} \verb"boaz.klartag@weizmann.ac.il, eli.putterman@weizmann.ac.il"


\end{document}